\documentclass[11pt]{amsart}

\usepackage{graphicx}
\usepackage{epsfig}

\usepackage{amsmath}
\usepackage{amssymb}
\usepackage{bigints}
\usepackage{tikz}
\usepackage{graphicx}
\usepackage{epsfig}
\usepackage[noadjust]{cite}

\usepackage{multicol}
\usepackage{xcolor}

\usetikzlibrary{shapes,arrows} 

\usetikzlibrary{decorations.markings}
\tikzset{
    set arrow inside/.code={\pgfqkeys{/tikz/arrow inside}{#1}},
    set arrow inside={end/.initial=>, opt/.initial=},
    /pgf/decoration/Mark/.style={
        mark/.expanded=at position #1 with
        {
            \noexpand\arrow[\pgfkeysvalueof{/tikz/arrow inside/opt}]{\pgfkeysvalueof{/tikz/arrow inside/end}}
        }
    },
    arrow inside/.style 2 args={
        set arrow inside={#1},
        postaction={
            decorate,decoration={
                markings,Mark/.list={#2}
            }
        }
    },
}

\theoremstyle{plain}
\newtheorem{theorem}                {Theorem}      [section]
\newtheorem*{theorem*}                {Theorem \ref{thm:appl}}
\newtheorem{proposition}  [theorem]  {Proposition}

\theoremstyle{definition}

\newtheorem{remark}       [theorem]  {Remark}

\DeclareMathOperator{\trace}{trace} 

\DeclareMathOperator{\Div}{div} 
 
\DeclareMathOperator{\ricci}{Ricci}

 \DeclareMathOperator{\Hess}{Hess}
 
\DeclareMathOperator{\grad}{grad}

\DeclareMathOperator{\Int}{Int}
\DeclareMathOperator{\Graf}{Graf}
\numberwithin{equation}{section}

\begin{document}

\title[On the uniqueness of complete biconservative surfaces in $\mathbb{R}^3$]
{On the uniqueness of complete biconservative surfaces in $\mathbb{R}^3$}

\author{Simona~Nistor}
\author{Cezar~Oniciuc}

\address{Faculty of Mathematics - Research Department\\ Al. I. Cuza University of Iasi\\
Bd. Carol I, 11 \\ 700506 Iasi, Romania} \email{nistor.simona@ymail.com}

\address{Faculty of Mathematics\\ Al. I. Cuza University of Iasi\\
Bd. Carol I, 11 \\ 700506 Iasi, Romania} \email{oniciucc@uaic.ro}




\begin{abstract}
We study the uniqueness of complete biconservative surfaces in the Euclidean space $\mathbb{R}^3$, and prove that the only complete biconservative regular surfaces in $\mathbb{R}^3$ are either $CMC$ or certain surfaces of revolution. In particular, any compact biconservative regular surface in $\mathbb{R}^3$ is a round sphere.
\end{abstract}

\maketitle
\section{Introduction}
In the last years, the theory of \textit{biconservative submanifolds} proved to be a very interesting research topic, as shown, for example, in \cite{CMOP14,FOP15,F15,FT16,MOR16,N16,N17,NO17,S12,S15,T15,UT16}. This topic arose from the theory of \textit{biharmonic submanifolds}, but the class of biconservative submanifolds is richer than the former one.

In order to define biharmonic submanifolds, let us first consider $\left(M^m,g\right)$ and $\left(N^n,h\right)$ two Riemannian manifolds and the \textit{bienergy functional}
$$
E_2:C^{\infty}(M,N)\rightarrow\mathbb{R},\quad E_{2}(\varphi)=\frac{1}{2}\int_{M}|\tau(\varphi)|^{2}\ v_g,
$$
where $\tau(\varphi)=\trace_g\nabla d\varphi$ is the \textit{tension field} of a smooth map $\varphi:M\to N$ with respect to the fixed metrics $g$ and $h$. A \textit{biharmonic map}, as suggested by J. Eells and J.H. Sampson in \cite{ES64}, is a critical point of the bienergy functional. The corresponding Euler-Lagrange equation, obtained by G.Y. Jiang in \cite{J86}, is
$$
\tau_{2}(\varphi)=-\Delta^{\varphi}\tau(\varphi)-\trace_g R^N(d\varphi,\tau(\varphi))d\varphi=0,
$$
where $\tau_{2}(\varphi)$ is called the \textit{bitension field} of $\varphi$, $\Delta^{\varphi}=-\trace_{g}
\left(\nabla^{\varphi}\nabla^{\varphi}-\nabla^{\varphi}_{\nabla}\right)$ is the rough Laplacian defined on sections of $\varphi^{-1}(TN)$, and $R^N$ is the curvature tensor field of $N$ given by
$$
R^N(X,Y)Z=[\nabla^N_X,\nabla^N_Y]Z-\nabla^N_{[X,Y]}Z.
$$
An isometric immersion $\varphi:\left(M^m,g\right)\to \left(N^n,h\right)$ or, simply, a submanifold $M$ of $N$, is called biharmonic if $\varphi$ is a biharmonic map. Any harmonic map is biharmonic and, therefore, we are interested in studying biharmonic and non-harmonic maps which are called \textit{proper-biharmonic}. As a submanifold $M$ of $N$ is minimal if and only if the map $\varphi:(M,g)\to (N,h)$ is harmonic, by a proper-biharmonic submanifold we mean a non-minimal biharmonic submanifold.

Following D. Hilbert (\cite{H24}), to an arbitrary functional $E$ we can associate a symmetric tensor field $S$ of type $(1,1)$, called the \textit{stress-energy tensor}, which is conservative, i.e., $\Div S=0$, at the critical points of $E$. In the particular case of the bienergy functional $E_2$, G.Y. Jiang (\cite{J87}) defined the stress-bienergy tensor $S_2$ by
\begin{align*}
\langle S_2(X),Y\rangle=&\frac{1}{2}|\tau(\varphi)|^2\langle X,Y\rangle+\langle d\varphi,\nabla\tau(\varphi)\rangle\langle X,Y\rangle\\&-\langle d\varphi(X),\nabla_Y\tau(\varphi)\rangle-\langle d\varphi(Y),\nabla_X\tau(\varphi)\rangle,
\end{align*}
and proved that
$$
\Div S_2=-\langle\tau_2(\varphi),d\varphi\rangle.
$$
Therefore, if $\varphi$ is biharmonic, i.e., is a critical point of $E_2$, then $\Div S_2=0$. The variational meaning of $S_2$ was explained in \cite{LMO08}.

One can see that if $\varphi:\left(M^m,g\right)\to\left(N^n,h\right)$ is an isometric immersion, then $\Div S_2=0$ if and only if the tangent part of the bitension field associated to $\varphi$ vanishes. A submanifold $M$ is called biconservative if $\Div S_2=0$.

The surfaces with constant non-zero mean curvature function ($CMC$ surfaces) and the minimal surfaces in $3$-dimensional spaces with constant sectional curvature $N^3(c)$ are trivially biconservative, and the explicit local parametric equations of biconservative surfaces in $\mathbb{R}^3$, $\mathbb{S}^3$ and $\mathbb{H}^3$, with $\grad f$ nowhere vanishing, where $f$ denotes the mean curvature function, were determined in \cite{CMOP14} and \cite{F15}. When the ambient space is $\mathbb{R}^3$ the result in \cite{HV95} was also reobtained in \cite{CMOP14}.

Even if the notion of biconservative submanifolds belongs to the extrinsic geometry, in the particular case of biconservative surfaces in $N^3(c)$, they also admit an intrinsic characterization that was found in \cite{FNO16}.

In \cite{N16} we extended the local classification results for biconservative surfaces in $N^3(c)$, with $c=0$ and $c=1$, to global results, i.e., we asked the biconservative surfaces to be complete and with $\grad f\neq 0$ on a non-empty open subset (non-$CMC$ surfaces). More precisely, we constructed, from the intrinsic and extrinsic point of view, complete biconservative surfaces in $\mathbb{R}^3$ and $\mathbb{S}^3$ with $\grad f$ nowhere vanishing on an open dense subset.

The paper is organized as follows. After a short section where we recall some results on biconservative surfaces in $\mathbb{R}^3$, we give, in the third section, some uniqueness results concerning complete non-$CMC$ biconservative surfaces in $\mathbb{R}^3$. We prove that the only complete non-$CMC$ biconservative regular surfaces in $\mathbb{R}^3$ are the known ones (see Theorem \ref{th:nonCMCbicons}) and that there exists no compact non-$CMC$ biconservative regular surface in $\mathbb{R}^3$ (see Theorem \ref{th:main-result-subsect}).

\vspace{0.2cm}

\noindent {\bf Convention.} All manifolds are assumed to be smooth, connected and oriented.

\section{Preliminaries}

As we have already seen, a biconservative submanifold $\varphi:M^m\to N^n$ is characterized by $\tau_2(\varphi)^\top=0$. If we consider the case of hypersurfaces $M^m$ in $N^{m+1}$, we get that $M$ is biconservative if and only if
$$
2A(\grad f)+f\grad f-2f(\ricci^N(\eta))^{\top}=0,
$$
where $\eta$ is the unit normal of $M$ in $N$, $A$ is the shape operator, $f=\trace A$ is the mean curvature function, and $(\ricci^N(\eta))^{\top}$ is the tangent component of the Ricci curvature of $N$ in the direction of $\eta$ (see ~\cite{LMO08,O10}).

Moreover, if $\varphi:M^m\to N^{m+1}(c)$ is a hypersurface, then it is biconservative if and only if
\begin{equation}\label{eq:bicons}
A(\grad f)=-\frac{f}{2}\grad f.
\end{equation}
It is then easy to see that any $CMC$ or minimal hypersurface in $N^{m+1}(c)$ is biconservative and, therefore, when studying biconservative surfaces in space forms we are interested in the non-$CMC$ case. Here, by a non-$CMC$ hypersurface we mean a hypersurface with $\grad f$ different from zero at some points but there could be points where $\grad f$ vanishes.

A direct consequence of ~\eqref{eq:bicons} is the following proposition.

\begin{proposition}\label{cor:eqf}
Let $M^m$ be a biconservative hypersurface in $N^{m+1}(c)$. Then
$$
f\Delta f-3|\grad f|^2-2\langle A,\Hess f\rangle=0.
$$
\end{proposition}

If we restrict our study to biconservative surfaces in $N^3(c)$ with $\grad f\neq 0$ at any point of $M$, we have the following local result.

\begin{theorem}[\cite{CMOP14}]\label{th:f-n3c}
Let $\varphi:M^2\to N^3(c)$ be a biconservative surface with $\grad f$ nowhere vanishing. Then, we have $f>0$ and
\begin{equation}\label{f-bicons}
f\Delta f+|\grad f|^2+\frac{4}{3}c f^2-f^4=0,
\end{equation}
where $\Delta$ is the Laplace-Beltrami operator on $M$.
\end{theorem}

Next, we recall some details about the construction of complete biconservative surfaces in $\mathbb{R}^3$ with $\grad f\neq 0$ at any point of an on an open dense subset, that was presented in \cite{MOR16,N16,NO17}. First, this construction is done from the extrinsic point of view, and then, from the intrinsic point of view. When using the extrinsic approach, we work, basically, with the images. In the intrinsic approach we work with immersions. In this case, the domain of the biconservative immersion is an abstract surface which is complete and simply connected.

We begin with the following \textit{local extrinsic} result which provides a characterization of biconservative surfaces in $\mathbb{R}^3$.

\begin{theorem}[\cite{HV95}]
Let $M^2$ be a surface in $\mathbb{R}^3$ with $(\grad f)(p)\neq0$ for any $p\in M$. Then, $M$ is biconservative if and only if, locally, it is a surface of revolution, and the curvature $\kappa=\kappa(u)$ of the profile curve $\sigma=\sigma(u)$, $\left|\sigma^\prime(u)\right|=1$, is a positive solution of the following $ODE$
$$
\kappa^{\prime\prime}\kappa=\frac{7}{4}\left(\kappa^\prime\right)^2-4\kappa^4.
$$
\end{theorem}

In \cite{CMOP14} there was found the local explicit parametric equation of a biconservative surface in $\mathbb{R}^3$.

\begin{theorem}[\cite{CMOP14}]\label{th3.3}
Let $M^2$ be a biconservative surface in $\mathbb{R}^3$ with $(\grad f)(p)\neq0$ for any $p\in M$. Then, locally, the surface can be parametrized by
$$
X_{\tilde{C}_0}(\rho,v)=\left(\rho\cos v,\rho \sin v, u_{\tilde{C}_0}(\rho)\right),
$$
where
$$
u_{\tilde{C}_0}(\rho)=\frac{3}{2\tilde{C}_0}\left(\rho^{1/3}\sqrt{\tilde{C}_0\rho^{2/3}-1}+\frac{1}{\sqrt{\tilde{C}_0}}\log\left(\sqrt{\tilde{C}_0}\rho^{1/3}+\sqrt{\tilde{C}_0\rho^{2/3}-1}\right)\right)
$$
with $\tilde{C}_0$ a positive constant and $\rho\in\left({\tilde{C}_0}^{-3/2},\infty\right)$.
\end{theorem}

We denote by $S_{\tilde{C}_0}$ the image $X_{\tilde{C}_0}\left(\left({\tilde{C}_0}^{-3/2},\infty\right)\times\mathbb{R}\right)$. We note that any two such surfaces $S_{\tilde{C}_0}$  and $S_{\tilde{C}^{'}_0}$, $\tilde{C}_0\neq \tilde{C}^{'}_0$, are not locally isometric, so we have a one-parameter family of biconservative surfaces in $\mathbb{R}^3$. These surfaces are not complete, but they are regular surfaces in $\mathbb{R}^3$ since the immersions defining them are embeddings.

We define the ``boundary'' of $S_{\tilde{C}_0}$ by $\overline{S}_{\tilde{C}_0}\setminus S_{\tilde{C}_0}$, where $\overline{S}_{\tilde{C}_0}$ is the closure of $S_{\tilde{C}_0}$ in $\mathbb{R}^3$.

The boundary of $S_{\tilde{C}_0}$ is the circle $\left({\tilde{C}_0}^{-3/2}\cos v,{\tilde{C}_0}^{-3/2}\sin v,0\right)$, which lies in the $Oxy$ plane. At a boundary point, the tangent plane to $\overline{S}_{\tilde{C}_0}$ is parallel to $Oz$. Moreover, along the boundary, the mean curvature function is constant $f_{\tilde{C}_0}=\left(2{\tilde{C}_0}^{3/2}\right)/3$ and $\grad f_{\tilde{C}_0}=0$; at any point of $S_{\tilde{C}_0}$, $f$ is positive and $\grad f$ is different from zero.

In \cite{N16,NO17} we also proved the following two properties that will be useful latter on.

\begin{proposition}[\cite{N16}]\label{prop:along-curve}
Let $S_{\tilde{C}_0}$ and $S_{\tilde{C}^\prime_0}$. Assume that we can glue them along a curve at the level of $C^\infty$ smoothness. Then $S_{\tilde{C}_0}$ and $S_{\tilde{C}^\prime_0}$ coincide or one of them is the symmetric of the another with respect to the plane where the common boundary lies.
\end{proposition}

As a consequence of the above result we have

\begin{theorem}[\cite{NO17}]\label{th:uniqueness-R3}
If $\varphi:M^2\to\mathbb{R}^3$ is a biconservative surface with $\grad f\neq 0$ at any point, then there exists a unique $\tilde{C}_0$ such that $\varphi(M)\subset S_{\tilde{C}_0}$.
\end{theorem}

Now we can state a \textit{global extrinsic} result.

\begin{theorem}[\cite{MOR16,N16}]\label{th-completeR3-1}
If we consider the symmetry of $\Graf u_{\tilde{C}_0}$, with respect to the $O\rho(=Ox)$ axis, we get a smooth complete biconservative surface $\tilde{S}_{\tilde{C}_0}$ in $\mathbb{R}^3$. Moreover, its mean curvature function $\tilde{f}_{\tilde{C}_0}$ is positive and $\grad \tilde{f}_{\tilde{C}_0}$ is different from zero at any point of an open dense subset of $\tilde{S}_{\tilde{C}_0}$.
\end{theorem}

\begin{remark}\label{rk:repraram-curve}
The profile curve $\sigma_{\tilde{C}_0}=\left(\rho,0,u_{\tilde{C}_0}(\rho)\right)\equiv \left(\rho,u_{\tilde{C}_0}(\rho)\right)$ can be re\-para\-metrized as
\begin{equation}\label{eq:sigma_tc}
\begin{array}{rl}
\sigma_{\tilde{C}_0}(\theta)= & \left(\sigma_{\tilde{C}_0}^1(\theta), \sigma_{\tilde{C}_0}^2(\theta)\right) \\\\
  = & {\tilde{C}_0}^{-3/2}\left((\theta+1)^{3/2},\frac{3}{2}\left(\sqrt{\theta^2+\theta} + \log \left(\sqrt{\theta}+\sqrt{\theta+1}\right)\right)\right), \qquad \theta>0,
\end{array}
\end{equation}
and then $X_{\tilde{C}_0}=X_{\tilde{C}_0}(\theta,v)$.
\end{remark}

\begin{remark}
The ``boundary'' of $S_{\tilde{C}_0}$ coincide with the boundary of $S_{\tilde{C}_0}$ as a subset of $\tilde{S}_{\tilde{C}_0}$, i.e., the intersection between the closure of $S_{\tilde{C}_0}$ in $\tilde{S}_{\tilde{C}_0}$ and the closure of $\tilde{S}_{\tilde{C}_0}\setminus S_{\tilde{C}_0}$ in $\tilde{S}_{\tilde{C}_0}$.
\end{remark}

\begin{remark}
The profile curves of the surfaces $S_{\tilde{C}_0}$ and $\tilde{S}_{\tilde{C}_0}$ are given in Figure \ref{FiguraProfileCurveR3}, for $\tilde{C}_0=9^{1/3}$.

\begin{figure}
\begin{tikzpicture}[xscale=1,yscale=1]

\node[inner sep=0pt]  at (2.2,1.2){\includegraphics[width=.3\textwidth]{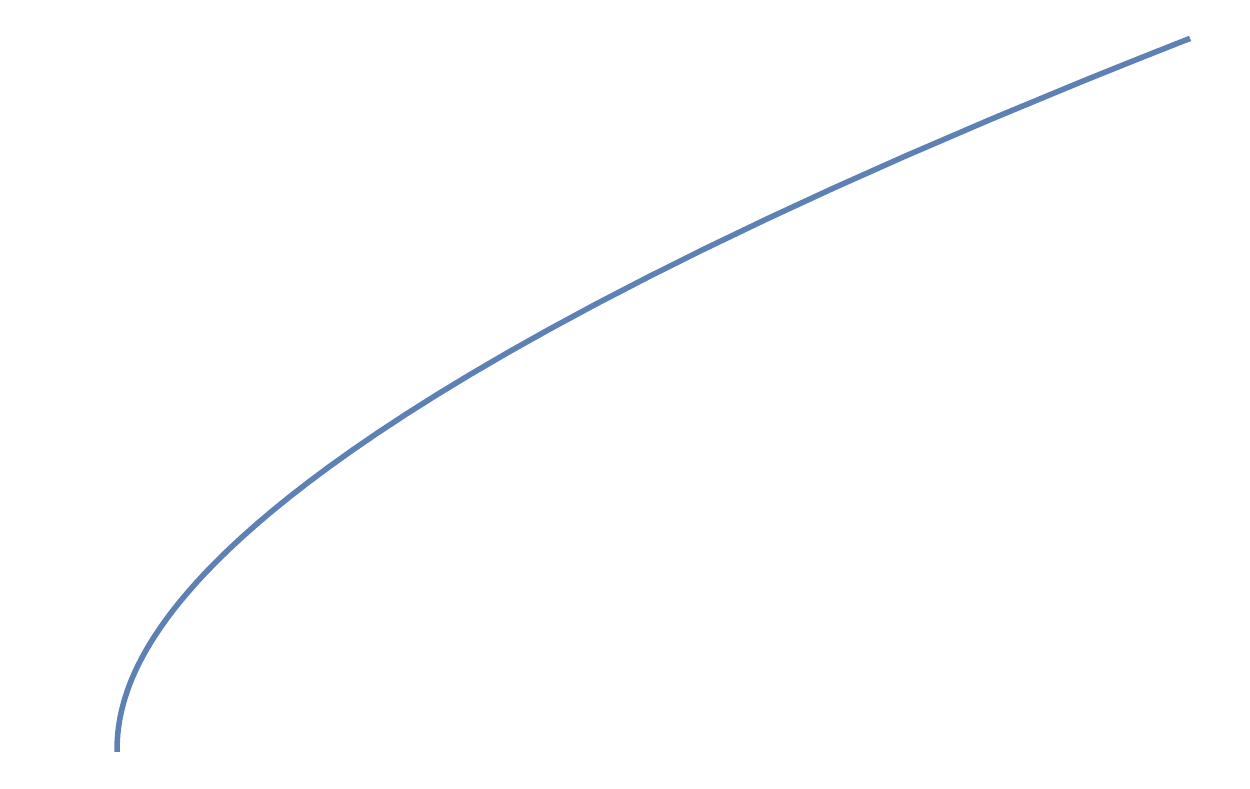}};
\draw [->,>=triangle 45] (-1,0) -- (5,0);
\draw [->,>=triangle 45] (0,-2) -- (0,4);
\draw (5,0) node[below] {$x$};
\draw (0,4) node[left] {$z$};
\end{tikzpicture}

\begin{tikzpicture}[xscale=1,yscale=1]

\node[inner sep=0pt]  at (2.2,0){\includegraphics[width=.3\textwidth]{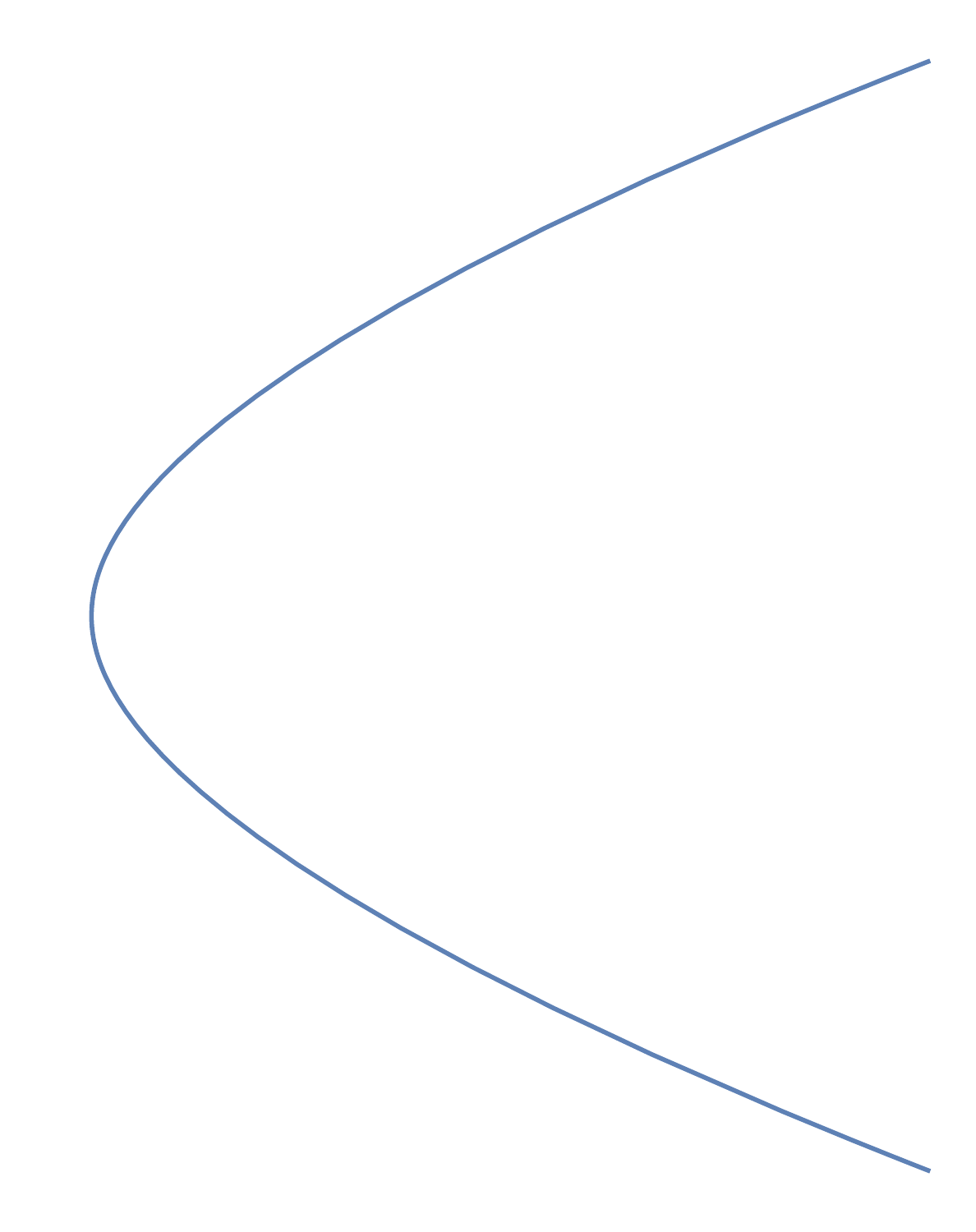}};
\draw [->,>=triangle 45] (-1,0) -- (5,0);
\draw [->,>=triangle 45] (0,-2) -- (0,4);
\draw (5,0) node[below] {$x$};
\draw (0,4) node[left] {$z$};
\end{tikzpicture}
\caption{The profile curves of $S_{\tilde{C}_0}$ and $\tilde{S}_{\tilde{C}_0}$} \label{FiguraProfileCurveR3}
\end{figure}
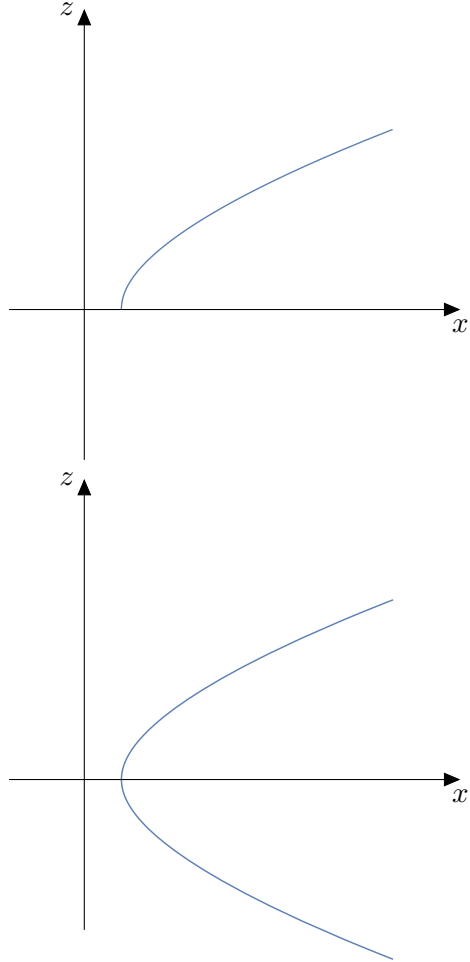

\end{remark}

Now, we change the point of view and construct, from the intrinsic point of view, complete biconservative surfaces in $\mathbb{R}^3$ with $\grad f\neq 0$ on an open dense subset.

First, we recall the \textit{local intrinsic} characterization for biconservative surfaces in $N^3(c)$.

\begin{theorem}[\cite{FNO16}]\label{thm:char}
Let $(M^2,g)$ be an abstract surface and $c\in\mathbb{R}$ a constant. Then $M$ can be locally isometrically embedded in a space form $N^3(c)$ as a biconservative surface with the gradient of the mean curvature function different from zero at any point of $M$ if and only if the Gaussian curvature $K$ satisfies $c-K(p)>0$, $(\grad K)(p)\neq 0$, for any $p\in M$, and its level curves are circles in $M$ with constant curvature
$$
\kappa=\frac{3|\grad K|}{8(c-K)}.
$$
\end{theorem}

Using local isothermal coordinates, we could find some more intrinsic characterizations of biconservative surfaces in $N^3(c)$.  One of them is given by the following result.

\begin{theorem}[\cite{N16}]\label{thm:char}
Let $\left(M^2,g\right)$ be an abstract surface and $c\in\mathbb{R}$ a constant. Then $M$ can be locally isometrically embedded in a space form $N^3(c)$ as a biconservative surface with the gradient of the mean curvature function different from zero at any point of $M$ if and only if the Gaussian curvature $K$ satisfies $c-K(p)>0$, $(\grad K)(p)\neq 0$, for any $p\in M$, and the metric $g$ can be locally written as $g=e^{2\rho}\left(du^2+dv^2\right)$, where $(u,v)$ are local coordinates positively oriented, and $\rho=\rho(u)$ satisfies the equation
\begin{equation}\label{ec-d}
\rho^{\prime\prime}=e^{-2\rho/3}-ce^{2\rho}
\end{equation}
and the condition $\rho^\prime>0$. Moreover, the solutions of the above equation, $u=u(\rho)$, are
$$
u=\int_{\rho_0}^{\rho}\frac{d\tau}{\sqrt{-3e^{-2\tau/3}-ce^{2\tau}+a}}+u_0,
$$
where $\rho$ is in some open interval $I$, $\rho_0\in I$ and $a,u_0\in \mathbb{R}$ are constants.
\end{theorem}

Then, from the above theorem for $c=0$, one obtains

\begin{proposition}[\cite{N16}]\label{prop4.3}
Let $\left(M^2,g\right)$ be an abstract surface. Then $M$ can be locally isometrically embedded in $\mathbb{R}^3$ as a biconservative surface with $\grad f\neq 0$ at any point of $M$ if and only if the Riemannian metric $g$ can be locally written as
$$
g_{C_0}(u,v)=C_0\left(\cosh u\right)^6(du^2+dv^2), \qquad u>0,
$$
where $C_0\in \mathbb{R}$ is a positive constant.
\end{proposition}

Concerning the complete non-$CMC$ biconservative surfaces in $\mathbb{R}^3$, we have the next \textit{global intrinsic} result.

\begin{theorem}[\cite{N16}]\label{main_th1}
Let $\left(\mathbb{R}^2,g_{C_0}=C_0 \left(\cosh u\right)^6\left(du^2+dv^2\right)\right)$ be a surface, where $C_0\in\mathbb{R}$ is a positive constant. Then we have:
\begin{itemize}
\item[(i)] the metric on $\mathbb{R}^2$ is complete;
\item[(ii)] the Gaussian curvature is given by
           $$
            K_{C_0}(u,v)=K_{C_0}(u)=-\frac{3}{C_0\left(\cosh u\right)^8}<0,\quad  K^\prime_{C_0}(u)=\frac{24 \sinh u}{C_0\left(\cosh u\right)^9},
           $$
           and therefore $\grad K_{C_0}\neq 0$ at any point of $\mathbb{R}^2\setminus Ov$;

\item[(iii)] the immersion $\varphi_{C_0}:\left(\mathbb{R}^2,g_{C_0}\right)\to \mathbb{R}^3$ given by
    $$
    \varphi_{C_0}(u,v)=\left(\sigma_{C_0}^1(u)\cos (3v), \sigma_{C_0}^1(u)\sin (3v), \sigma_{C_0}^2(u)\right)
    $$
    is biconservative and it is an embedding, where
    $$
    \sigma_{C_0}^1(u)=\frac{\sqrt{C_0}}{3}\left(\cosh u\right)^3, \quad
    \sigma_{C_0}^2(u)=\frac{\sqrt{C_0}}{2}\left(\frac{1}{2}\sinh (2u)+u\right), \qquad u \in \mathbb{R}.
    $$
\end{itemize}
\end{theorem}

\begin{remark}
For the above immersion, $\grad f_{C_0}\neq 0$ at any point of $\mathbb{R}^2\setminus Ov$.
\end{remark}

\section{Uniqueness of biconservative surfaces in $\mathbb{R}^3$}\label{subsec:uniqueness}

In Theorem \ref{main_th1}, $\grad f$ is nowhere zero on an open dense subset. Various attempts to construct a biconservative surface in $\mathbb{R}^3$ such that $\grad f=0$ on a subset with non-empty interior and $\grad f\neq 0$ on a non-empty subset failed. This led us to the following conjecture.

\vspace{0.2cm}
\noindent\textbf{Conjecture 1.} \textit{Let $M^2$ be a biconservative surface in $\mathbb{R}^3$. If there exists an non-empty open subset $U$ of $M$ such that $\grad f=0$ on $U$, then $\grad f=0$ on $M$.}
\vspace{0.2cm}

We can also propose a stronger statement.

\vspace{0.2cm}
\noindent\textbf{Conjecture 2.} \textit{The only simply connected complete non-$CMC$ biconservative surfaces in $\mathbb{R}^3$ are those given by Theorem \ref{main_th1}.}
\vspace{0.2cm}

To make things simpler, we could work only with regular surfaces in $\mathbb{R}^3$, i.e. with those surfaces defined by embeddings, and basically we would work only with images instead of maps. In this case, the corresponding statement to the Conjecture 2 is

\vspace{0.2cm}
\noindent\textbf{Conjecture 3.} \textit{The only complete non-$CMC$ biconservative regular surfaces in $\mathbb{R}^3$ are the surfaces $\tilde{S}_{\tilde{C}_0}$, $\tilde{C}_0 > 0$.}
\vspace{0.2cm}

In Theorem \ref{th:contradiction} and Theorem \ref{th:2connectedcomp} we prove Conjecture 1 and, respectively, Conjecture 2, under some some additional hypotheses. Then, in Theorem \ref{th:nonCMCbicons}, we prove Conjecture 3.

Even if the uniqueness of complete non-$CMC$ biconservative surfaces was already mentioned in ~\cite{N16,NO17}, a rigorous approach of this problem has been done only here.

\vspace{0.2cm}

Our first result is the following theorem.

\begin{theorem}\label{th:contradiction}
Let $\varphi:M^2\to\mathbb{R}^3$ be a non-$CMC$ surface. Assume that
$$
W=\left\{p\in M\ |\ (\grad f)(p)\neq 0\right\}
$$
is connected, $M\setminus W$ has non-empty interior and the boundaries in $M$ of $\Int(M\setminus W)$ and $M\setminus W$ coincide, i.e., $\partial^M \Int(M\setminus W)=\partial^M (M\setminus W)$. Then $M$ cannot be biconservative.
\end{theorem}

\begin{proof}
Assume that $M$ is biconservative. The boundary $\partial^MW$ is non-empty and let us consider an arbitrary point
$$
p_0\in \partial^M W=\partial^M(M\setminus W)=\partial^M\Int(M\setminus W).
$$
As $W$ is an open subset, it follows that $p_0\notin W$.

There exists an open subset $U_0$ of $M$ such that $p_0\in U_0$ and $\varphi_{|_{U_0}}:U_0\to\mathbb{R}^3$ is an embedding. Thus, we can identify $U_0$ with its image $\varphi\left(U_0\right)\subset\mathbb{R}^3$ and then $U_0$ can be seen as a regular surface in $\mathbb{R}^3$.

We note that $p_0\in\partial^M W \cap U_0$ allows the existence of a sequence $\left(p^1_n\right)_{n\in\mathbb{N}^\ast}\subset W\cap U_0$, $p_n^1\neq p_0$, for any $n\in \mathbb{N}^\ast$, which converges to $p_0$, with respect to the intrinsic distance function $d_M$ on $M$, and, similarly, from $p_0\in \partial^M \Int(M\setminus W)\cap U_0$ it follows that there exists a sequence $\left(p^2_n\right)_{n\in\mathbb{N}^\ast}\subset \Int (M\setminus W)\cap U_0$, $p_n^2\neq p_0$, for any $n\in \mathbb{N}^\ast$, which converges to $p_0$, with respect to $d_M$. It is easy to see that we can identify $p^1_n=\varphi\left(p^1_n\right)$ and $p^2_n=\varphi\left(p^2_n\right)$.

Now, since $W$ is connected, open in $M$ and $\grad f\neq 0$ at any point of $W$, from Theorem \ref{th:uniqueness-R3}, one obtains that there exists a unique $\tilde{C}_0$ such that $\varphi(W)$ is an open subset in $S_{\tilde{C}_0}$. Then, as $W\cap U_0$ is open in $W$, it is clear that $\varphi\left(W\cap U_0\right)$ is open in $S_{\tilde{C}_0}$. In fact, using the identification $\varphi\left(W\cap U_0\right)=W\cap U_0$, we have that $W\cap U_0$ is open in $S_{\tilde{C}_0}$. We recall that $S_{\tilde{C}_0}$ is open in the complete surface $\tilde{S}_{\tilde{C}_0}$, and then $W\cap U_0$ is also open in $\tilde{S}_{\tilde{C}_0}$.

Further, we denote by $d_0$ the distance function on $\mathbb{R}^3$ and by $d_{\tilde{S}_{\tilde{C}_0}}$ the intrinsic distance function on $\tilde{S}_{\tilde{C}_0}$.

Obviously, the convergence of the sequence $\left(p^1_n\right)$ to $p_0$ with respect to the distance $d_M$, implies the convergence of $\left(p^1_n\right)$ to $p_0$ with respect to $d_0$.

The sequence $\left(p^1_n\right)$ has been chosen in $W\cap U_0$, so $\left(p^1_n\right)\subset\tilde{S}_{\tilde{C}_0}$. As $\tilde{S}_{\tilde{C}_0}$ is a closed subset in $\mathbb{R}^3$, we obtain that $p_0\in \tilde{S}_{\tilde{C}_0}$ and thus $\left(p^1_n\right)$ converges to $p_0$ also with respect to the distance $d_{\tilde{S}_{\tilde{C}_0}}$.

We have already seen that $W\cap U_0$ is open in $U_0$ and also in $\tilde{S}_{\tilde{C}_0}$. Then, the mean curvature functions $f$ and $\tilde{f}_{\tilde{C}_0}$ corresponding to $M$ and $\tilde{S}_{\tilde{C}_0}$, respectively, coincide on $W\cap U_0$. Therefore, for any $n\in \mathbb{N}^\ast$, one has
\begin{equation}\label{eq:3ec}
\left\{
\begin{array}{l}
  f\left(p^1_n\right)=\tilde{f}_{\tilde{C}_0}\left(p^1_n\right) \\
  (\grad f)\left(p^1_n\right)=\left(\grad \tilde{f}_{\tilde{C}_0}\right)\left(p^1_n\right) \\
  \left|(\grad f)\left(p^1_n\right)\right|=\left|\left(\grad \tilde{f}_{\tilde{C}_0}\right)\left(p^1_n\right)\right| \\
  (\Delta f)\left(p^1_n\right)=\left(\Delta \tilde{f}_{\tilde{C}_0}\right)\left(p^1_n\right)
\end{array}
\right..
\end{equation}
From the convergence of $\left(p^1_n\right)$ to the same $p_0$, with respect to both distance functions $d_M$ and $d_{\tilde{S}_{\tilde{C}_0}}$, and from the third equation in \eqref{eq:3ec}, one gets
$$
\left|(\grad f)\left(p_0\right)\right|=\left|\left(\grad \tilde{f}_{\tilde{C}_0}\right)\left(p_0\right)\right|.
$$
As $p_0\notin W$, we have $(\grad f)\left(p_0\right)=0$ and then $\left(\grad \tilde{f}_{\tilde{C}_0}\right)\left(p_0\right)=0$, which means that $p_0$ belongs to the boundary of $S_{\tilde{C}_0}$ in $\tilde{S}_{\tilde{C}_0}$. Thus, $\tilde{f}_{\tilde{C}_0}\left(p_0\right)=2\tilde{C}_0^{3/2}/3\neq 0$.

In particular, we get that $\partial^M W\subset\partial^{\tilde{S}_{\tilde{C}_0}} S_{\tilde{C}_0}$.

Using the first equation in \eqref{eq:3ec} and  the convergence of $\left(p^1_n\right)$ to $p_0$, with respect to $d_M$ and $d_{\tilde{S}_{\tilde{C}_0}}$, we obtain $f\left(p_0\right)=\tilde{f}_{\tilde{C}_0}\left(p_0\right)$.

From \eqref{f-bicons} for $c=0$, one has
$$
f\left(p^1_n\right)\left(\Delta f\right)\left(p^1_n\right)+\left|\left(\grad f\right)\left(p^1_n\right)\right|^2-f^4\left(p^1_n\right)=0,
$$
for any $n\in \mathbb{N^\ast}$.
We may pass to the limit with respect to the distance $d_M$ in the above equation and obtain
$$
f\left(p_0\right)\left(\Delta f\right)\left(p_0\right)+\left|\left(\grad f\right)\left(p_0\right)\right|^2-f^4\left(p_0\right)=0.
$$
According to the above facts, this is equivalent to
\begin{equation}\label{eq:contradiction}
\tilde{f}_{\tilde{C}_0}\left(p_0\right)\left(\Delta f\right)\left(p_0\right)-\tilde{f}_{\tilde{C}_0}^4\left(p_0\right)=0.
\end{equation}
We have also seen that there exists a sequence $\left(p^2_n\right)_{n\in\mathbb{N}^\ast}\subset \Int (M\setminus W)\cap U_0$ which converges to $p_0$, with respect to the distance $d_M$. Since $\grad f=0$ at any point of $\Int (M\setminus W)\cap U_0$ and $\Int (M\setminus W)\cap U_0$ is open in $M$, it is easy to see that $(\grad f)\left(p^2_n\right)=0$ and $(\Delta f)\left(p^2_n\right)=0$, for any $n\in \mathbb{N}^\ast$. Considering  the limit with respect to the distance $d_M$ in the above two relations we get $(\grad f)\left(p_0\right)=0$ and $(\Delta f)\left(p_0\right)=0$.

Substituting $(\Delta f)\left(p_0\right)=0$ in \eqref{eq:contradiction}, one obtains $\tilde{f}_{\tilde{C}_0}\left(p_0\right)=0$ and we come to a contradiction, which concludes the proof.
\end{proof}

\begin{theorem}\label{th:2connectedcomp}
Let $\varphi:M^2\to\mathbb{R}^3$ be a biconservative surface. Assume that
$$
W=\left\{p\in M\ |\ (\grad f)(p)\neq 0\right\}
$$
is dense and it has two connected components, $W_1$ and $W_2$. Assume that the boundaries of $W_1$ and $W_2$ in $M$ coincide and their common boundary is a smooth curve in $M$. Then, there exists a unique $\tilde{C}_0$ such that $\varphi(M)\subset \tilde{S}_{\tilde{C}_0}$. Moreover, if $M$ is complete and simply connected, then up to isometries of the domain and codomain, $\varphi$ is the map given in Theorem \ref{main_th1}.
\end{theorem}

\begin{proof}
Let us consider $p_0\in \partial^M W_1=\partial^M W_2$. There exists an open subset $U_0$ in $M$, such that $p_0\in U_0$ and $\varphi_{|_{U_0}}:U_0\to\mathbb{R}^3$ is an embedding. Thus, we can identify $U_0=\varphi\left(U_0\right)\subset\mathbb{R}^3$.

Since $W_1$ and $W_2$ are connected, open in $M$ and $\grad f\neq 0$ at any point of them, from Theorem \ref{th:uniqueness-R3}, one obtains that there exist $\tilde{C}_0$ and $\tilde{C}_0^\prime$ such that $\varphi\left(W_1\right)$ is an open subset of $S_{\tilde{C}_0}$ and $\varphi\left(W_2\right)$ is an open subset of $S_{\tilde{C}_0^\prime}$.

It is easy to see that $U_0\cap W_1$ is open in $W_1$ and then $\varphi\left(U_0\cap W_1\right)=U_0\cap W_1$ is open in $S_{\tilde{C}_0}$. Analogously, $U_0\cap W_2$ is open in $W_2$ and then $\varphi\left(U_0\cap W_2\right)=U_0\cap W_2$ is open in $S_{\tilde{C}_0^\prime}$.

We note that, as $W$ is the reunion of two open disjoint subsets, and as $W$ is dense in $M$, one has
\begin{align*}
\partial^M W & = \partial^M W_1\cup \partial^M W_2=\partial^M W_1=\partial^M W_2 \\
& = M\setminus W
\end{align*}
and
$$
M=W\cup\partial^M W=W_1\cup W_2\cup \partial^M W_1.
$$
Therefore
\begin{align*}
  U_0 & = U_0\cap M \\
  & = \left(U_0\cap W_1\right)\cup \left(U_0\cap W_2\right)\cup \left(U_0\cap \partial^M W_1\right).
\end{align*}
We consider $U_0\cap\partial^M W_1$ as the image of a smooth curve $\gamma:I\to U_0$, $\gamma^\prime(s)\neq 0$, for any $s\in I$. It is clear from the proof of Theorem ~\ref{th:contradiction} that $\gamma(s)\in \tilde{S}_{\tilde{C}_0}$ and $\left(\grad\tilde{f}_{\tilde{C}_0}\right)(\gamma(s))=(\grad f)(\gamma(s))$. Further, the fact that $(\grad f)(\gamma(s))=0$ will follow as in the first part of the proof of Theorem ~\ref{th:W0C0} (for the simplicity of exposition we do not include it here); $\gamma(s)\notin W_1$ but now $W_1$ does not contain all the points where $\grad f$ is different from zero (as in Theorem ~\ref{th:contradiction}). Therefore, $\gamma(s)$ belongs to the boundary of $S_{\tilde{C}_0}$ in $\tilde{S}_{\tilde{C}_0}$, for any $s\in I$, which is a circle of radius $\tilde{C}_0^{-3/2}$. Using the same argument, we see that $\gamma(s)$ belongs to the boundary of $S_{\tilde{C}_0^\prime}$ in $\tilde{S}_{\tilde{C}_0^\prime}$, for any $s\in I$, which is a circle of radius $\tilde{C}_0^{\prime{-3/2}}$. Therefore, $\tilde{C}_0=\tilde{C}_0^\prime$ and $\varphi(M)\subset \tilde{S}_{\tilde{C}_0}$.

If $M$ is complete, $\varphi:M\to\tilde{S}_{\tilde{C}_0}$ is a covering space with the projection $\varphi$ and thus $\varphi(M)=\tilde{S}_{\tilde{C}_0}$. Moreover, if $M$ is also simply connected, then $M$ is a universal covering of $\tilde{S}_{\tilde{C}_0}$ with the projection $\varphi$.

The map $\varphi_{C_0}:\left(\mathbb{R}^2,g_{C_0}\right)\to\tilde{S}_{\tilde{C}_0}$ given in Theorem \ref{main_th1} is also a universal covering projection and therefore there exists an isometry $\Theta$ between $(M,g)$ and $\left(\mathbb{R}^2,g_{C_0}\right)$ such that $\varphi_{C_0}\circ \Theta=\varphi$.
\end{proof}

\vspace{0.2cm}

In the following, we restrict our study to the case when $\varphi:M^2\to\mathbb{R}^3$ is an embedding, i.e., $S=\varphi(M)$ is a regular surface in $\mathbb{R}^3$. The next result is the main ingredient for proving Conjecture 3.

\begin{theorem}\label{th:W0C0}
Let $S$ be a complete non-$CMC$ biconservative regular surface in $\mathbb{R}^3$. Denote by
$$
W=\left\{p\in S \ | \ (\grad f)(p)\neq0\right\}
$$
and by $W_0$ a connected component of $W$. Then, there exists a unique $\tilde{C}_0>0$ such that $W_0=S_{\tilde{C}_0}$. Moreover, the closure of $W_0$ in $S$ coincides with the closure of $S_{\tilde{C}_0}$ in $\tilde{S}_{\tilde{C}_0}$.
\end{theorem}

\begin{proof}
We note that $W_0$ is closed and also open in $W$, as $W_0$ is a connected component of $W$, and $W$ is an open subset of $S$. It is clear that $W_0$ is also open in $S$.

We denote by $\partial^S W_0$ the boundary of $W_0$ in $S$. It is non-empty, as if it was, $W_0$ would be also closed in $S$, so $W_0=S$. But this implies $S\subset S_{\tilde{C}_0}$, for some $\tilde{C}_0$, which contradicts the completeness of $S$.

We will prove that $(\grad f)\left(q\right)=0$, for any $q\in \partial^S W_0$. To do this, we assume that there exists a point $q_0\in \partial^S W_0$ such that $(\grad f)\left(q_0\right)\neq0$. Then, it follows that one has an open ball $B^2\left(q_0;r_0\right)$ in $S$, $r_0>0$, such that $\grad f$ is different from zero at any point.

Obviously, $q_0$ belongs to the closure of $W_0$ in $S$, and then $B^2\left(q_0;r_0\right)\cap W_0\neq \emptyset$. Since $B^2\left(q_0;r_0\right)$ and $W_0$ are connected sets, we get that $B^2\left(q_0;r_0\right)\cup W_0$ is also connected.  Moreover, as $\grad f$ is different from zero at any point of $B^2\left(q_0;r_0\right)$, it follows that $B^2\left(q_0;r_0\right)\subset W$ and, therefore $B^2\left(q_0;r_0\right)\cup W_0$ is a connected subset of $W$. Now, from the maximality of $W_0$ in $W$ with respect to the inclusion, we have $B^2\left(q_0;r_0\right)\cup W_0=W_0$, i.e., $B^2\left(q_0;r_0\right)\subset W_0$. Obviously, $q_0\in W_0$, and this cannot be true because $q_0\in \partial^S W_0$ and $W_0$ is open $S$.

Thus, one has $(\grad f)\left(q\right)=0$, for any $q\in \partial^S W_0$.

From Theorem \ref{th:uniqueness-R3}, since $W_0$ is connected and $\grad f\neq 0$ at any point of $W_0$, one obtains that there exists a unique $\tilde{C}_0$ such that $W_0$ is open in $S_{\tilde{C}_0}$. Moreover, we will prove that, as $S$ is complete, $W_0=S_{\tilde{C}_0}$.

Let us consider $\sigma_{\tilde{C}_0}:(0,\infty)\to \mathbb{R}^2$ the profile curve of $S_{\tilde{C}_0}$, $\sigma_{\tilde{C}_0}(0,\infty)\subset S_{\tilde{C}_0}$. We can reparametrize $\sigma_{\tilde{C}_0}$ by arc-length, such that the new curve, denoted also by $\sigma_{\tilde{C}_0}=\sigma_{\tilde{C}_0}(\theta)$, has the same orientation as the initial one, is defined on $(0,\infty)$ and in zero has the same limit point $\left(\tilde{C}_0^{-3/2},0\right)$ on the boundary of $S_{\tilde{C}_0}$. The new curve $\sigma_{\tilde{C}_0}$ is a parametrized geodesic of $S_{\tilde{C}_0}$ and $\left(\grad f_{\tilde{C}_0}\right)\left(\sigma_{\tilde{C}_0}(\theta)\right)\neq 0$, for any $\theta>0$.

Next, we will prove that $\sigma_{\tilde{C}_0}(0,\infty)\subset W_0$. Clearly, there exists a point $\theta_0\in (0,\infty)$ such that $\sigma_{\tilde{C}_0}\left(\theta_0\right)\in W_0$. Since $\sigma_{\tilde{C}_0}$ is continuous and $W_0$ is open in $S_{\tilde{C}_0}$, it follows that exists $\varepsilon_0>0$ such that $\left(\theta_0-\varepsilon_0,\theta_0+\varepsilon_0\right)\subset (0,\infty)$ and
$$
\sigma_{\tilde{C}_0}\left(\theta_0-\varepsilon_0,\theta_0+\varepsilon_0\right)\subset W_0.
$$
Assume that $\sigma_{\tilde{C}_0}(0,\infty)\not\subset W_0$, i.e., there exists $\theta'\in (0,\infty)\setminus\left(\theta_0-\varepsilon_0,\theta_0+\varepsilon_0\right)$ such that $\sigma_{\tilde{C}_0}\left(\theta'\right)\not\in W_0$.

Assume that $\theta'\geq \theta_0+\varepsilon_0$. Denote
$$
\Omega=\left\{ \theta \ |\ \theta>\theta_0, \sigma_{\tilde{C}_0}(\theta)\not\in W_0\right\}
$$
and $\theta_1=\inf \Omega$. Of course, $\theta_1\geq \theta_0+\varepsilon_0$ and $\sigma_{\tilde{C}_0}(\theta)\in W_0$ for any $\theta \in \left[\theta_0,\theta_1\right)$.

Next, we will see that $\sigma_{\tilde{C}_0}\left(\theta_1\right)\not\in W_0$. Indeed, if $\sigma_{\tilde{C}_0}\left(\theta_1\right)\in W_0$, it follows that there exists $\varepsilon_1>0$ such that $\left(\theta_1-\varepsilon_1,\theta_1+\varepsilon_1\right)\subset (0,\infty)$ and $\sigma_{\tilde{C}_0}\left(\theta_1-\varepsilon_1,\theta_1+\varepsilon_1\right)\subset W_0$. Therefore we obtain a contradiction because, in this case, $\theta_1$ cannot be the infimum of $\Omega$.

Since $\sigma_{\tilde{C}_0}(\theta)\in W_0$ for any $\theta \in \left[\theta_0,\theta_1\right)$, it is clear that $\sigma_{\tilde{C}_0}\left(\theta_1\right)$ belongs to the closure of $W_0$ in $S_{\tilde{C}_0}$, denoted by $\overline{W_0}^{S_{\tilde{C}_0}}$. As $\sigma_{\tilde{C}_0}\left(\theta_1\right)\not\in W_0$ and $W_0$ is open in $S_{\tilde{C}_0}$, one obtains that $\sigma_{\tilde{C}_0}\left(\theta_1\right)\in \partial^{S_{\tilde{C}_0}} W_0$, i.e. $\sigma_{\tilde{C}_0}\left(\theta_1\right)$ belongs to the boundary of $W_0$ in $S_{\tilde{C}_0}$.

We have seen that $\sigma_{\tilde{C}_0}(\theta)\in W_0$, for any $\theta\in \left(\theta_0-\varepsilon_0,\theta_1\right)$. Now, since $W_0$ is open in $S$ and $\sigma_{\tilde{C}_0}$ is a parametrized geodesic of $S_{\tilde{C}_0}$, we get that $\sigma_{\tilde{C}_0}$ defined on $\left(\theta_0-\varepsilon_0,\theta_1\right)$ is also a parametrized geodesic of $S$. As $S$ is complete, we can consider the parametrized geodesic $\sigma^S$ of $S$ defined on the whole $\mathbb{R}$ such that $\left.\sigma^S\right|_{\left(\theta_0-\varepsilon_0,\theta_1\right)}= \left.\sigma_{\tilde{C}_0}\right|_{\left(\theta_0-\varepsilon_0,\theta_1\right)}$. Since $\sigma^S$ and $\sigma_{\tilde{C}_0}$ are continuous, we must have $\sigma^S\left(\theta_1\right)=\sigma_{\tilde{C}_0}\left(\theta_1\right)$. Now, we note that $\sigma^S\left(\theta_1\right)\in \partial^S W_0$, because $\sigma^S\left(\theta_1\right)\in \overline{W_0}^{S}$ and $\sigma^S\left(\theta_1\right)=\sigma_{\tilde{C}_0}\left(\theta_1\right)\not\in W_0$. Therefore,
\begin{equation}\label{eq:grad=0}
(\grad f)\left(\sigma^S\left(\theta_1\right)\right)=(\grad f)\left(\sigma_{\tilde{C}_0}\left(\theta_1\right)\right)=0.
\end{equation}

We know that $W_0$ is open in both $S$ and $S_{\tilde{C}_0}$, thus the mean curvature function of $S$ coincides with the mean curvature function of $S_{\tilde{C}_0}$ on $W_0$, i.e., $\left. f\right|_{W_0}=\left. f_{\tilde{C}_0}\right|_{W_0}$, and their gradients are also equal, i.e., $\left.\left(\grad f\right)\right|_{W_0}=\left.\left(\grad f_{\tilde{C}_0}\right)\right|_{W_0}$. Clearly,
$$
(\grad f)\left(\sigma^S(\theta)\right)=\left(\grad f_{\tilde{C}_0}\right)\left(\sigma_{\tilde{C}_0}(\theta)\right), \qquad \theta\in \left(\theta_0-\varepsilon_0,\theta_1\right),
$$
and, then,
$$
\left|(\grad f)\left(\sigma^S(\theta)\right)\right|=\left|\left(\grad f_{\tilde{C}_0}\right)\left(\sigma_{\tilde{C}_0}(\theta)\right)\right|, \qquad \theta\in \left(\theta_0-\varepsilon_0,\theta_1\right).
$$
We take the limit in the above relation and obtain
$$
\left|(\grad f)\left(\sigma^S\left(\theta_1\right)\right)\right|=\left|\left(\grad f_{\tilde{C}_0}\right)\left(\sigma_{\tilde{C}_0}\left(\theta_1\right)\right)\right|.
$$
Using \eqref{eq:grad=0}, one gets a contradiction, because $\left(\grad f_{\tilde{C}_0}\right)\left(\sigma_{\tilde{C}_0}\left(\theta_1\right)\right)\neq0$, as $\sigma_{\tilde{C}_0}\left(\theta_1\right)\in S_{\tilde{C}_0}$.

Thus, $\sigma_{\tilde{C}_0}(\theta)\in W_0$, for any $\theta \geq \theta_0+\varepsilon_0$.

In the same way, we can prove that $\sigma_{\tilde{C}_0}(\theta)\in W_0$, also for any $\theta \leq \theta_0-\varepsilon_0$.

Finally, we obtain that $\sigma_{\tilde{C}_0}(0,\infty)\subset W_0$ and, then $\sigma_{\tilde{C}_0}=\left.\sigma^S\right|_{(0,\infty)}$.

Now, we recall that $S_{\tilde{C}_0}$ is open in $\tilde{S}_{\tilde{C}_0}$, and then $W_0$ is also open in $\tilde{S}_{\tilde{C}_0}$. Since $\tilde{S}_{\tilde{C}_0}$ is complete, we can consider the parametrized geodesic $\tilde{\sigma}_{\tilde{C}_0}$ of $\tilde{S}_{\tilde{C}_0}$ defined on the whole $\mathbb{R}$ such that $\left.\tilde{\sigma}_{\tilde{C}_0}\right|_{\left(0,\infty\right)}= \sigma_{\tilde{C}_0}=\left.\sigma^S\right|_{\left(0,\infty\right)}$. Obviously, $\sigma^S(0)=\tilde{\sigma}_{\tilde{C}_0}(0)$, thus
$\left.\tilde{\sigma}_{\tilde{C}_0}\right|_{\left[0,\infty\right)}=\left.\sigma^S\right|_{\left[0,\infty\right)}$ and the closure of $W_0$ in $S$ is included in the closure of $S_{\tilde{C}_0}$ in $\tilde{S}_{\tilde{C}_0}$.

Further, we consider $\gamma$ the curve parametrized by arc-length, defined on the whole $\mathbb{R}$, which gives the boundary of $S_{\tilde{C}_0}$ in $\tilde{S}_{\tilde{C}_0}$. Clearly, $\gamma$ is a parametrized geodesic of $\tilde{S}_{\tilde{C}_0}$. According to the above observations, it follows that there exist $a,b\in\mathbb{R}$ with $a<b$, such that $\gamma(a,b)$ belongs also to $S$. Moreover, $\gamma$ defined on $(a,b)$ is also a geodesic of $S$ because, along it, the normal vector field to $S$ coincide with the normal vector field to $\tilde{S}_{\tilde{C}_0}$ and it is collinear with the principal unit normal vector of $\left.\gamma\right|_{(a,b)}$. Since $S$ is complete, we can consider $\gamma^S:\mathbb{R}\to S$ the parametrized geodesic of $S$ such that $\left.\gamma^S\right|_{(a,b)}=\left.\gamma\right|_{(a,b)}$.

We note that the maximal interval which contains $(a,b)$ and has the property that its image by $\gamma$ is contained in $S$, is $\mathbb{R}$. Indeed, assume that there exists $b'$, $b\leq b'<\infty$, such that $\gamma(a,b')\subset S$ and $\gamma(b')\notin S$. It is now easy to see that, as $\left.\gamma^S\right|_{(a,b')}=\left.\gamma\right|_{(a,b')}$, $\gamma(b')=\gamma^S(b')\in S$, and thus we get a contradiction.

Therefore, $W_0=S_{\tilde{C}_0}$.
\end{proof}

\begin{remark}
The proof of Theorem \ref{th:W0C0} can be summarized as in Figure \ref{Figura14}, where the yellow region represents the connected component $W_0$ of $W$ in $S$, and the surface of revolution represented in green is the corresponding $S_{\tilde{C}_0}$ (given by Theorem \ref{th:uniqueness-R3}). It is suggested that, in fact, all the meridians of $S_{\tilde{C}_0}$ which intersect $W_0$ are contained in $W_0$ and then, as the boundary of $W_0$ in $S$ has to be the whole circle which gives the boundary of $S_{\tilde{C}_0}$ in $\tilde{S}_{\tilde{C}_0}$, $W_0=S_{\tilde{C}_0}$.

\begin{figure}
\begin{tikzpicture}
\node [inner sep=0pt] at (-1,0) {\includegraphics[scale=0.38]{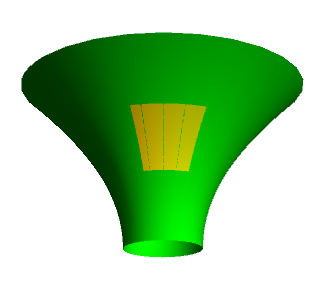}};
\draw[->,thick](1,0)-- (5.5,0);
\node [inner sep=0pt] at (8,0) {\includegraphics[scale=0.38]{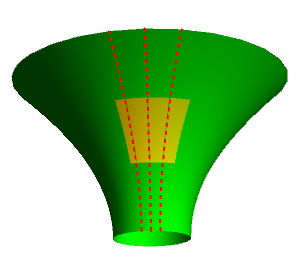}};
\draw[->,thick](8,-2)-- (8,-4);
\node [inner sep=0pt] at (8,-6) {\includegraphics[scale=0.4]{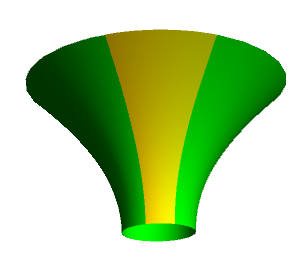}};
\draw[->,thick](5.5,-6)-- (1.2,-6);
\node [inner sep=0pt] at (-1,-6) {\includegraphics[scale=0.38]{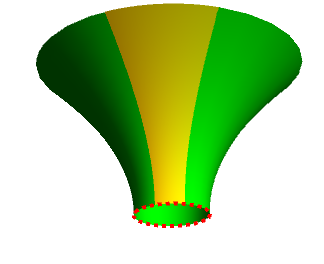}};
\draw[-,thick](-0.9,-7.5)-| (-0.9,-8.5);
\draw[-,thick](-0.9,-8.5)-| (3,-8.5);
\draw[->,thick](3,-8.5)-| (3,-9.5);
\node [inner sep=0pt] at (3.5,-11.5) {\includegraphics[scale=0.4]{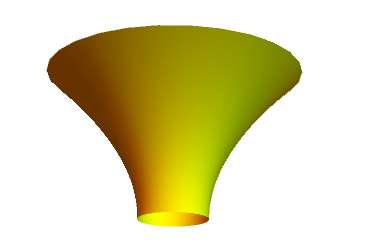}};
\end{tikzpicture}
\caption{The idea of the proof of Theorem \ref{th:W0C0}} \label{Figura14}
\end{figure}
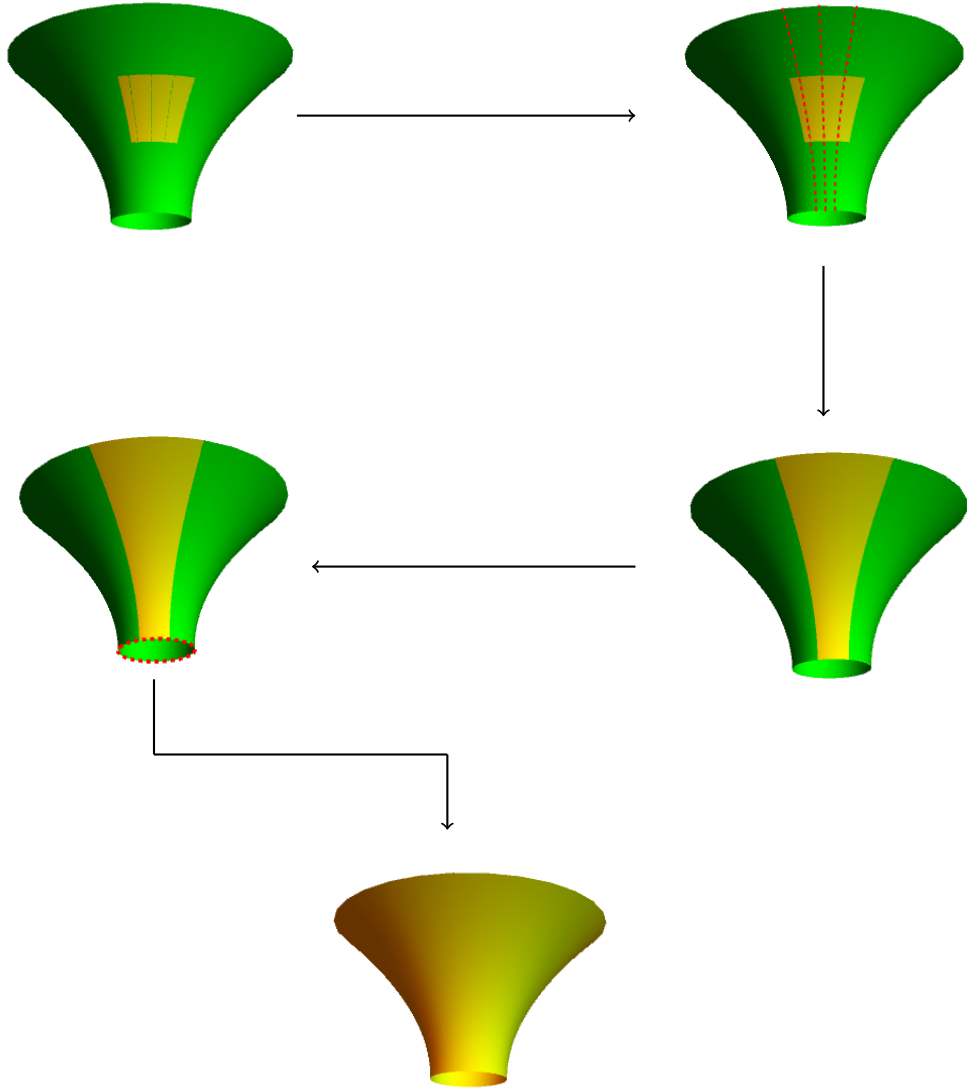
\end{remark}

A first consequence of the above result is the following theorem.

\begin{theorem}\label{th:main-result-subsect}
Let $S$ be a biconservative regular surface in $\mathbb{R}^3$. Assume that $S$ is compact. Then $S$ is $CMC$, and therefore a round sphere.
\end{theorem}

\begin{proof}
Assume that $S$ is non-$CMC$, i.e., $W=\left\{p\in S\ | \ (\grad f)(p)\neq0\right\}$ is a non-empty open subset of $S$. Let us consider $W_0$ a connected component of $W$. From Theorem \ref{th:W0C0}, it follows that there exists $\tilde{C}_0>0$ such that $W_0=S_{\tilde{C}_0}$ and the closure of $W_0$ in $S$ coincides with the closure of $S_{\tilde{C}_0}$ in $\tilde{S}_{\tilde{C}_0}$. Since $S_{\tilde{C}_0}$ is unbounded in $\mathbb{R}^3$, we obtain that $S$ is unbounded too, and this is a contradiction as $S$ is compact.

The last part follows from a well-known result of Alexandrov (see, for example \cite{MR05}).
\end{proof}

Another consequence of Theorem \ref{th:W0C0} is the next result, which is a counter part of Theorem ~\ref{th:contradiction} for regular surfaces.

\begin{theorem}\label{th:nonbicons}
Let $S$ be a complete non-CMC regular surface in $\mathbb{R}^3$. Assume that
$$
W=\left\{p\in S\ | \ (\grad f)(p)\neq0\right\}
$$
is connected. Then $S$ cannot be biconservative.
\end{theorem}

\begin{proof}
Assume that $S$ is biconservative. From Theorem \ref{th:W0C0} it follows that there exists $\tilde{C}_0>0$ such that $W=S_{\tilde{C}_0}$ and the closure of $W$ in $S$ coincides with the closure of $S_{\tilde{C}_0}$ in $\tilde{S}_{\tilde{C}_0}$.

It is easy to see that $W$ is not dense in $S$, because if we assume so, then $S=\overline{W}^S=\overline{S}_{\tilde{C}_0}^{\tilde{S}_{\tilde{C}_0}}$. This would mean that $S$ is a surface with boundary, fact that contradicts the regularity of $S$. Therefore, $S\setminus\overline{W}^S$ is non-empty and open in $S$. We note that, as
$$
S\setminus W=\overline{S\setminus W}^S=\overline{S\setminus\overline{W}},
$$
we get
$$
\partial^SW=\partial^S\overline{W}^S=\partial^S(S\setminus{\overline{W}^S}).
$$

Further, let us consider $p_0\in \partial^S W=\partial ^S \left(S\setminus\overline{W}^S\right)$, which is the circle of radius $\tilde{C}_0^{-3/2}$. Then there exists a sequence $\left(p^1_n\right)_{n\in\mathbb{N}^\ast}$ in $W$, with $p^1_n\neq p_0$, for any $n\in\mathbb{N}^\ast$ which converges to $p_0$, with respect to the distance function $d_S$ on $S$, and another sequence $\left(p^2_n\right)_{n\in\mathbb{N}^\ast}$ in $S\setminus\overline{W}^S$, with $p^2_n\neq p_0$, for any $n\in\mathbb{N}^\ast$ which converges to $p_0$, with respect to the same $d_S$. Since $W$ and $S\setminus \overline{W}^S$ are open in $S$, $\grad f$ is different from zero at any point of $W$, and $\grad f$, $\Delta f$ vanish at any point of $S\setminus \overline{W}^S$, we can use the same argument as in the proof of Theorem \ref{th:contradiction} to obtain a contradiction.

Therefore, our assumption is false, and $S$ cannot be biconservative.
\end{proof}

Now we can state the main result.

\begin{theorem}\label{th:nonCMCbicons}
Let $S$ be a complete non-$CMC$ biconservative regular surface in $\mathbb{R}^3$. Then $S=\tilde{S}_{\tilde{C}_0}$.
\end{theorem}

\begin{proof}
We denote by $W_0$ a connected component of the non-empty open subset
$$
W=\left\{p\in S \ | \ (\grad f)(p)\neq0\right\}.
$$
Then there exists a unique $\tilde{C}_0$ such that $W_0=S_{\tilde{C}_0}$ and $\overline{W_0}^{S}=\overline{S_{\tilde{C}_0}}^{\tilde{S}_{\tilde{C}_0}}$. We denote, the surface $S_{\tilde{C}_0}$ by $S_{\tilde{C}_0}^+$. Let $p_0\in \partial^S W_0$, i.e., $p_0$ is a point on a circle of radius $\tilde{C}_0^{-3/2}$. We have three cases.

First, assume that there exists $\varepsilon_0>0$ such that $\grad f$ vanishes at any point of $B^2\left(p_0;\varepsilon_0\right)\setminus \left(B^2\left(p_0;\varepsilon_0\right)\cap \overline{W_0}^S\right)$. Then, there exists a sequence $\left(p^1_n\right)_{n\in\mathbb{N}^\ast}$ in $W_0$, with $p^1_n\neq p_0$, for any $n\in\mathbb{N}^\ast$ which converges to $p_0$, with respect to the distance function $d_S$ on $S$, and another sequence $\left(p^2_n\right)_{n\in\mathbb{N}^\ast}$ in $B^2\left(p_0;\varepsilon_0\right)\setminus \left(B^2\left(p_0;\varepsilon_0\right)\cap \overline{W_0}^S\right)$, with $p^2_n\neq p_0$, for any $n\in\mathbb{N}^\ast$ which converges to $p_0$, with respect to the same $d_S$. Since $W_0$ and $B^2\left(p_0;\varepsilon_0\right)\setminus \left(B^2\left(p_0;\varepsilon_0\right)\cap \overline{W_0}^S\right)$ are open in $S$, $\grad f$ is different from zero at any point of $W_0$ and $\grad f$ vanishes at any point of $B^2\left(p_0;\varepsilon_0\right)\setminus \left(B^2\left(p_0;\varepsilon_0\right)\cap \overline{W_0}^S\right)$, we obtain a contradiction as in the proof of Theorem \ref{th:contradiction}.

Second, let us assume that there exists $\varepsilon_0>0$ such that $\grad f$ is different from zero at any point of $B^2\left(p_0;\varepsilon_0\right)\setminus \left(B^2\left(p_0;\varepsilon_0\right)\cap \overline{W_0}^S\right)$. As the set $B^2\left(p_0;\varepsilon_0\right)\setminus \left(B^2\left(p_0;\varepsilon_0\right)\cap \overline{W_0}^S\right)$ is connected, it belongs to a connected component $\tilde{W}_0$ of $W$. It is clear that $\tilde{W}_0$ has to coincide with $S_{\tilde{C}_0}^-$, where $S_{\tilde{C}_0}^-$ is the surface obtained by symmetry of $S_{\tilde{C}_0}^+$ with respect to the plane where its boundary lies, and thus $\tilde{S}_{\tilde{C}_0}\subset S$. Since $\tilde{S}_{\tilde{C}_0}$ is complete, then it cannot be extendible, so $\tilde{S}_{\tilde{C}_0}=S$.

In the last case, assume that for any $\varepsilon_n>0$, in $B^2\left(p_0;\varepsilon_n\right)\setminus \left(B^2\left(p_0;\varepsilon_n\right)\cap \overline{W_0}^S\right)$ there exists at least a point $p^1_{n}$ such that $(\grad f)\left(p^1_{n}\right)=0$ and at least a point $p^2_n$ such that $(\grad f)\left(p^2_n\right)\neq0$.

Let us consider an arbitrary $\varepsilon_1>0$. Then there exists $U_1$ an open subset of $S$ which contains $p^2_1$, which is connected, $U_1\subset B^2\left(p_0;\varepsilon_1\right)\setminus \left(B^2\left(p_0;\varepsilon_1\right)\cap \overline{W_0}^S\right)$ and $\grad f$ does not vanish at any point of $U_1$. If we consider the connected component of $W$ which contains $U_1$, one can notice that this is a surface $S_{\tilde{C}_0^1}\subset S$ and $\overline{S_{\tilde{C}_0^1}}^S=\overline{S_{\tilde{C}_0^1}}^{\tilde{S}_{\tilde{C}_0^1}}$. Moreover, $S_{\tilde{C}_0^1}\cap S_{\tilde{C}_0}^+=\emptyset$.

We note that the boundaries of any two connected components of $W$ either coincide or are disjoint.

Also, it is easy to see that the boundary of $S_{\tilde{C}_0^1}$ does not intersect the boundary of $S_{\tilde{C}_0}^+$. Indeed, if these two boundaries would intersect, they would coincide and $S_{\tilde{C}_0^1}=S_{\tilde{C}_0}^-$. Therefore, $S=\tilde{S}_{\tilde{C}_0}$ and we come to a contradiction as the gradient of the mean curvature function is different from zero at any point of $S_{\tilde{C}_0}^-$ but $(\grad f)(p^1_n)=0$.

Next, we can consider $\varepsilon_2>0$ such that $B^2\left(p_0;\varepsilon_2\right)\setminus \left(B^2\left(p_0;\varepsilon_2\right)\cap \overline{W_0}^S\right)$ does not intersect the boundary of $S_{\tilde{C}_0^1}$. Using the same argument, one obtains another surface $S_{\tilde{C}_0^2}\subset S$ such that $S_{\tilde{C}_0^2}\cap S_{\tilde{C}_0}^+=\emptyset$ and $S_{\tilde{C}_0^2}\cap S_{\tilde{C}_0^1}=\emptyset$. Moreover, $\overline{S_{\tilde{C}_0^2}}^{\tilde{S}_{\tilde{C}_0^2}}\cap \overline{S_{\tilde{C}_0}^+}^{\tilde{S}_{\tilde{C}_0}}=\emptyset$ and $\overline{S_{\tilde{C}_0^2}}^{\tilde{S}_{\tilde{C}_0^2}}\cap \overline{S_{\tilde{C}_0^1}}^{\tilde{S}_{\tilde{C}_0^1}}=\emptyset$. We continue in the same way and obtain a sequence of surfaces $\left(S_{\tilde{C}_0^n}\right)_{n\in\mathbb{N}^\ast}\subset S$, which are mutually disjoint and $S_{\tilde{C}_0^n}\cap S^+_{\tilde{C}_0}=\emptyset$, for any $n\in\mathbb{N}^\ast$. Their boundaries are mutually disjoint too and they do not intersect the boundary of $S_{\tilde{C}_0}^+$. For any $n\in \mathbb{N}^\ast$, the boundary of $S_{\tilde{C}_0^n}$ is a circle of radius $\left(\tilde{C}_0^n\right)^{-3/2}=2/\left(3\tilde{f}_{\tilde{C}_0^n}\right)$, where $\tilde{f}_{\tilde{C}_0^n}$ is the mean curvature function of $\tilde{S}_{\tilde{C}_0^n}$ evaluated at some point of the boundary of $S_{\tilde{C}_0^n}$. But $\tilde{f}_{\tilde{C}_0^n}=f$ on $\overline{S_{\tilde{C}_0^n}}^{\tilde{S}_{\tilde{C}_0^n}}$, where $f$ is the mean curvature function of $S$. As $f$ is continuous, the sequence $\left(\tilde{C}_0^n\right)^{-3/2}$ converges to the radius of the circle that gives the boundary of $S_{\tilde{C}_0}^+$, which is $\left(\tilde{C}_0\right)^{-3/2}=2/\left(3\tilde{f}_{\tilde{C}_0}\right)$, where $\tilde{f}_{\tilde{C}_0}$ is the mean curvature function of $\tilde{S}_{\tilde{C}_0}$, $\tilde{f}_{\tilde{C}_0}=f$ on $\overline{S_{\tilde{C}_0}}^{\tilde{S}_{\tilde{C}_0}}$, evaluated at some point of the boundary of $S_{\tilde{C}_0}$.

Since $S$ is a regular surface, we cannot have such a sequence $\left(S_{\tilde{C}_0^n}\right)_{n\in\mathbb{N}^\ast}\subset S$.
\end{proof}

\end{document}